\documentclass{amsart}
\usepackage{amssymb, amsmath, latexsym, mathabx, dsfont,tikz}
\usepackage{multirow}

\renewcommand{\baselinestretch}{\baselinestretch}
\renewcommand{\baselinestretch}{1.1}
\numberwithin{equation}{section}

\newtheorem{thm}{Theorem}[section]
\newtheorem{lem}[thm]{Lemma}
\newtheorem{cor}[thm]{Corollary}

\theoremstyle{definition}
\newtheorem{defn}[thm]{Definition}

\theoremstyle{remark}
\newtheorem{rmk}[thm]{Remark}

\numberwithin{equation}{section}

\newcommand{\ra}{{\rightarrow}}

\newcommand{\gen}{\text{gen}}

\newcommand{\ord}{\text{ord}}

\newcommand{\z}{{\mathbb Z}}
\newcommand{\q}{{\mathbb Q}}

\newcommand{\e}{{\epsilon}}

\begin{document}
\title[Quadratic forms with a strong regularity property]{Quadratic forms with a strong regularity property on the representations of squares}

\author{Kyoungmin Kim and Byeong-Kweon Oh}

\address{Department of Mathematics, Sungkyunkwan University, Suwon 16419, korea}
\email{kiny30@skku.edu}
\thanks{This work of the first author was supported by the National Research Foundation of Korea(NRF) grant funded by the Korea government(MSIT) (NRF-2016R1A5A1008055  and NRF-2018R1C1B6007778)}

\address{Department of Mathematical Sciences and Research Institute of Mathematics, Seoul National University, Seoul 08826, Korea}
\email{bkoh@snu.ac.kr}
\thanks{This work of the second author was supported by the National Research Foundation of Korea (NRF-2017R1A2B4003758).}



\begin{abstract}    A (positive definite and non-classic  integral) quadratic form is called strongly $s$-regular if it satisfies a strong regularity property on the number of representations of squares of integers. In this article, we prove that  for any integer $k \ge 2$, there are only finitely many isometry classes of strongly $s$-regular quadratic forms with rank $k$ if the minimum of the nonzero squares that are represented by them is fixed. 
\end{abstract}

\maketitle

\section{Introduction}

For a positive definite (non-classic) integral quadratic form of rank $k$
$$ 
f(x_1,x_2,\dots,x_k)=\sum_{1 \le i \le j \le k}a_{ij}x_ix_j \quad (a_{ij} \in \z),
$$
we define $r(n,f)$ the number of representations of an integer $n$ by $f$, that is, 
$$
r(n,f)=\vert\{(x_1,x_2,\dots,x_k)\in \z^k\mid f(x_1,x_2,\dots,x_k)=n\}\vert.
$$
Hence $r(n,f)$ is the number of solutions of the quadratic diophantine equation. 
It is well known that $r(n,f)$ is finite for any positive integer $n$.  However,  there is no known method to compute this number $r(n,f)$ explicitly for an arbitrary quadratic form $f$.   

Hurwitz \cite {hur} noted that if $f(x,y,z)=x^2+y^2+z^2$, then 
$$
r(n^2,f)=6\prod_{p \,:\, \text{odd prime}} \left[ \frac{p^{\ord_p(n)+1}-1}{p-1}-\left(\frac{-1}p\right)\frac{p^{\ord_p(n)}-1}{p-1}\right].
$$  
A complete proof of this formula was given by Olds \cite{ol}.  Note that the product in the right hand side is, in fact, the product of primes dividing $n$, for the value is one if $\ord_p(n)=0$.  Motivated by Hurwitz's formula, Cooper and Lam \cite{cl}  proved that if $f$ is isometric to one of
$$
x^2+y^2+2z^2, \ \ x^2+y^2+3z^2, \ \  x^2+2y^2+2z^2, \  \ \text{and} \ \   x^2+3y^2+3z^2,
$$
then it satisfies a similar equation to Hurwitz's formula. Furthermore,  they  provided a list of $64$ diagonal ternary quadratic forms and  conjectured that  those ternary forms in the list also satisfy similar equations. Some parts of their conjecture were proved by  Guo, Peng, and Qin  \cite {gpq}, H\"urlimann  \cite{hu}, and  finally, all of the remaining cases  were proved by 
the authors \cite{ko1}.

To explain the condition given by Cooper and Lam more precisely, we briefly introduce some terminology on quadratic forms. 
For an integer $n$, we define
$$
w(f)=\sum_{[g] \in \gen(f)/\sim} \frac1{o(g)} \qquad \text{and} \qquad r(n,\gen(f))=\frac1{w(f)}\sum_{[g] \in \gen(f)/\sim} \frac{r(n,g)}{o(g)},
$$ 
where $\gen(f)/\sim$ is the set of isometry class $[g] \subset \gen(f)/\sim$, and $o(f)$ is the order of the isometry group $O(f)$.
Though there is no known method to compute $r(n,f)$ for an arbitrary quadratic form $f$,  the Minkowski-Siegel formula says that 
$r(n,\gen(f))$ is the constant multiple of the product of local densities. 

Roughly speaking, Cooper and Lam's condition is that $r(n^2,f)$ is almost same to $r(n^2,\gen(f))$ for any positive integer $n$. To be more precise, let $f$ be a positive definite (non-classic) integral quadratic form of rank $k$. Let $n_1$ and $n_2$ be positive integers such that $\mathcal P(n_1)\subset \mathcal P(2df)$, $(n_2,2df)=1$ and $n=n_1n_2$. Here  $\mathcal P(n)$ denotes the set of prime factors of $n$, and  $df$ denotes the discriminant of $f$ defined by the determinant of the symmetric matrix
$M_f=\left(\frac{\partial^2f}{\partial x_i\partial x_j}\right).$ 
The quadratic form $f$ is called {\it strongly $s$-regular} if for any positive integer $n=n_1n_2$ such that $r(n_1^2,f)\ne 0$,
\begin{equation} \label{eqn1}
\frac{r(n_1^2n_2^2,f)}{r(n_1^2,f)}= \frac{r(n_1^2n_2^2,\gen(f))}{r(n_1^2,\gen(f))}.
\end{equation}
The Minkowski-Siegel formula says that (for details, see Corollary \ref{sregular})
$$
\frac{r(n_1^2n_2^2,\gen(f))}{r(n_1^2,\gen(f))} =\prod_{p\mid n_2}h_p(df,\mu_p),
$$
where $\mu_p=\text{ord}_p(n_2)$ for any prime $p$ and
$$
h_p(df,\mu_p)\!=\!
\begin{cases}
\displaystyle \sum_{t=0}^{2\mu_p}\left(\frac{(-1)^{\frac{k}2}df}p\right)^tp^{\frac{(k-2)t}{2}}
\ \text{if $k$ is even},\\[0.5cm]

\displaystyle \!\left(\frac{p^{(k-2)(\mu_p +1)}-1}{p^{k-2}-1}-p^{\frac{k-3}{2}} \left(\frac{(-1)^{\frac{k-1}{2}}df}{p}\right)\frac{p^{(k-2)\mu_p}-1}{p^{k-2}-1}\right)\ \!\text{otherwise}.
\end{cases}
$$
Clearly, if $f$ does not represent any squares of integers, then  $f$ is trivially strongly $s$-regular. So, throughout this article, we always assume that any strongly $s$-regular  quadratic form $f$ represents at least one square of an integer. Note that this condition is equivalent to the condition that $f$ represents $1$ over $\q$.  A diagonal ternary quadratic form $f$ satisfies the Cooper and Lam's condition given in \cite{cl} if and only if $f$ satisfies Equation \eqref{eqn1}, that is, $f$ is strongly $s$-regular. 

As a natural generalization of the Cooper and Lam's conjecture, one may consider the problem to classify  all strongly $s$-regular quadratic forms. Related to this question,  we \cite{ko2} proved that  there are only finitely many isometry classes of strongly $s$-regular ternary quadratic forms if the minimum of the nonzero squares that are represented by the form is fixed, and we classified all strongly $s$-regular ternary quadratic forms that represent $1$.  Recently, the first author \cite{kkm}  proved that  there are only finitely many isometry classes of strongly $s$-regular quaternary quadratic forms if the minimum of the nonzero squares that are represented by them is fixed, and classifies all strongly $s$-regular diagonal quaternary quadratic forms that represent $1$.  

The aim of the article is to prove that  there are only finitely many isometry classes of strongly $s$-regular quadratic forms with fixed rank greater than $1$ if the minimum of the nonzero squares that are represented by them is fixed.

The subsequent discussion will be conducted in the better adapted geometric language of quadratic spaces and lattices.  Throughout this article, we always assume that every $\z$-lattice $L=\z x_1+\z x_2+\dots+\z x_k$ is {\it positive definite and non-classic integral}, that is,  the corresponding symmetric matrix 
$$
M_L=(2B(x_i,x_j)) \in M_{k\times k}(\z)
$$ 
is positive definite and  the norm ideal $\mathfrak n(L)$ is $\z$. We define $dL=\det(M_L)$. Note that the definition of  the discriminant $dL$ of a $\z$-lattice $L$ is different from that of \cite{om}. 
 If an integer $n$ is represented by $L$ over $\z_p$ for any prime $p$ including infinite prime, then we say that $n$ is represented by the genus of $L$, and we write $n \,\,\ra \,\,\gen(L)$.  
When $n$ is represented by the lattice $L$ itself, then we write $n\,\,\ra\,\, L$.   
We always assume that $\Delta_p$ is a non square unit in $\z_p^{\times}$ for any odd prime $p$. Recall the for any positive integer $n$, the set of prime factors of $n$ is denoted by $\mathcal P(n)$.

Any unexplained notation and terminology can be found in \cite{ki} or  \cite{om}.

\section{Almost strongly $s$-regular $\z$-lattices}
In this section, we define almost strongly $s$-regularities of quadratic forms, and provide some basic properties of those quadratic forms.

\begin{defn}\label{almost}
Let $P$ be a finite set of primes and let $L$ be a $\z$-lattice of rank $k$ that represents at least one square of integer.
For any positive integer $n$, let $n_1$ and $n_2$ be integers such that $n=n_1n_2$, and 
$$
\mathcal P(n_1)\subset \mathcal P(2dL)\cup P, \quad  \left(n_2,2dL\cdot\prod_{p\in P}p\right)=1.
$$
We say the $\z$-lattice $L$ is {\it almost strongly $s$-regular with respect to $P$} if for any positive integer $n=n_1n_2$,
$$
r(n_1^2n_2^2,L)=r(n_1^2,L)\cdot \prod_{p\mid n_2}h_p(dL,\mu_p),
$$
where $\mu_p=\text{ord}_p(n_2)$ for any prime $p$, and
$$
h_p(dL,\mu_p)\!=\!
\begin{cases}
\displaystyle \sum_{t=0}^{2\mu_p}\left(\frac{(-1)^{\frac{k}2}dL}p\right)^tp^{\frac{(k-2)t}{2}}
\ \text{if $k$ is even},\\[0.5cm]

\displaystyle \!\left(\frac{p^{(k-2)(\mu_p +1)}-1}{p^{k-2}-1}-p^{\frac{k-3}{2}} \!\left(\frac{(-1)^{\frac{k-1}{2}}dL}{p}\right)\!\frac{p^{(k-2)\mu_p}-1}{p^{k-2}-1}\right)\ \!\!\text{otherwise}.
\end{cases}
$$
For convenience, if there exists a finite set of primes $P$ such that $L$ is almost strongly $s$-regular with respect to $P$, then we simply say $L$ is {\it almost strongly $s$-regular}.  In particular, if $L$ satisfies the above  condition when $P=\emptyset$, then we  say that $L$ is {\it strongly $s$-regular}.
\end{defn}

In fact, one may easily check whether or not a $\z$-lattice $L$ represents at least one square of an integer.
\begin{lem} \label{local-s}
Let $L$ be a $\z$-lattice of rank $k$ and let $V=L\otimes \q$ be the quadratic space over $\q$. 
\begin{itemize}
\item [(i)]If $k=2$, then $r(n^2,L')=0$ for any $L'\in \gen(L)$ and any integer $n$ if and only if 
$S_p(V) \ne (-1,d(V_p))_p$  for some prime $p$.

\item [(ii)]If $k=3$, then $r(n^2,L')=0$ for any $L' \in \gen(L)$ and any integer $n$ if and only if $d(V_p)=-1$ and  $S_p(V) \ne (-1,-1)_p$  for some prime $p$.

\item [(iii)]If $k\ge 4$, then $L$ represents at least one square of an integer.
\end{itemize}
Here $d(V_p)$ is the discriminant of $V_p=V\otimes \q_p$, and $S_p(V)$ is the Hasse symbol of $V$ over $\q_p$
\end{lem}
\begin{proof}
The lemma follows directly from the fact that for any integer $n$, $r(n^2,L')=0$ for any $L'\in \gen(L)$ if and only if $1$ is not represented by $V$.
\end{proof}

\begin{lem} \label{tec} Let $P$ be a finite set of primes and let
 $L$ be an almost strongly $s$-regular $\z$-lattice with respect to $P$. Then we have
 $$ 
 r(n^2m^2,L)=r(n^2,L)\prod_{p\mid m} h_p(dL,\ord_p(m)),
 $$
  for any integers $n,m$ such that $(m, 2\cdot n\cdot dL\cdot \prod_{p \in P}p)=1$.
\end{lem}

\begin{proof} Let $n=n_1n_2$ for some integers $n_1$ and $n_2$ satisfying all conditions given in Definition \ref{almost}.  Since $(n_2,m)=1$,  we have
$$
\begin{array} {ll} 
\displaystyle r(n^2m^2,L)&=r(n_1^2n_2^2m^2,L)=r(n_1^2,L)\displaystyle \prod_{p\mid n_2m} h_p(dL,\ord_p(n_2m))\\
&=r(n_1^2,L)\displaystyle \prod_{p\mid n_2}h_p(dL,\ord_p(n_2))\prod_{p\mid m}h_p(dL,\ord_p(m))\\
&=r(n^2,L)\displaystyle \prod_{p\mid m}h_p(dL,\ord_p(m)).\\
\end{array}
$$
This completes the proof.
\end{proof}

\begin{cor} \label{almost1} Let $P_1$ and $P_2$ be sets of finite primes such that $P_1 \subseteq P_2$. 
Then any almost strongly $s$-regular $\z$-lattice with respect to $P_1$ is also almost strongly $s$-regular with respect to $P_2$. 
In particular, any strongly $s$-regular $\z$-lattice is almost strongly $s$-regular with respect to any finite set of primes. 
\end{cor}

\begin{proof} The corollary follows directly from Lemma \ref{tec}.
\end{proof}

\begin{lem} \label{min}
Let $L$ be a $\z$-lattice of rank $k$. For any positive integer $n$, let $n_1$ and $n_2$ be positive integers such that $\mathcal P(n_1)\subset \mathcal P(2dL)$, $(n_2,2dL)=1$ and $n=n_1n_2$. For any prime $p$, we put $\ord_p(n)=\mu_p$. 
\begin{itemize}
\item [(i)]Assume that $k$ is even. If $n_1$ is represented by the genus of $L$, then
$$
\begin{array}{ll}
\displaystyle\frac{r(n_1n_2,\gen(L))}{r(n_1,\gen(L))}
 &\displaystyle=n_2^{\frac{k-2}{2}}\cdot\prod_{p\mid 2dL}\frac{\alpha_p(n_1n_2,L)}{\alpha_p(n_1,L)}\prod_{p\nmid 2dL}\frac{\alpha_p(n_1n_2,L)}{\alpha_p(n_1,L)}\\[0.45cm]
 &\displaystyle=\prod_{p\mid n_2}\left[\sum_{t=0}^{\mu_p}\left(\frac{(-1)^{\frac{k}2}dL}p\right)^tp^{\frac{(k-2)t}{2}}\right].
\end{array}
$$
In particular, if $n_1^2$ is represented by the genus of $L$, then
$$
\frac{r(n_1^2n_2^2,\gen(L))}{r(n_1^2,\gen(L))}=\displaystyle\prod_{p\mid n_2}\left[\sum_{t=0}^{2\mu_p}\left(\frac{(-1)^{\frac{k}2}dL}p\right)^tp^{\frac{(k-2)t}{2}}\right].
$$
\item [(ii)]Assume that $k$ is odd. If $n_1^2$ is represented by the genus of $L$, then
$$
\begin{array}{ll}
 \displaystyle \frac{r(n_1^2n_2^2,\gen(L))}{r(n_1^2,\gen(L))} \!\!\!&=\displaystyle n_2^{k-2}\cdot\prod_{p\mid 2dL}\frac{\alpha_p(n_1^2n_2^2,L)}{\alpha_p(n_1^2,L)}\prod_{p\nmid 2dL}\frac{\alpha_p(n_1^2n_2^2,L)}{\alpha_p(n_1^2,L)}\\[0.45cm]
 \!\!\!&\displaystyle=\prod_{p\mid n_2}\left(\frac{p^{(k-2)(\mu_p +1)}-1}{p^{k-2}-1}-p^{\frac{k-3}{2}} \left(\frac{(-1)^{\frac{k-1}{2}}dL}{p}\right)\frac{p^{(k-2)\mu_p}-1}{p^{k-2}-1}\right).
\end{array}
$$
\end{itemize}
In particular, any $\z$-lattice with class number $1$ is strongly $s$-regular. 
\end{lem}
\begin{proof}
By the Minkowski-Siegel formula, we have
$$
r(n,\gen(L))=\e_k \cdot\pi^{\frac k2}\cdot \Gamma\left(\frac k2\right)^{-1}\cdot (dL)^{-\frac 12} \cdot n^{\frac{k-2}{2}}\cdot \prod_{p<\infty}\alpha_p(n,L),
$$
where $\e_k=\frac12$ if $k=2$, $\e_k=1$ otherwise, and $\alpha_p$ is the local density.
Assume that $p$ does not divide $2dL$. If $k$ is even, then
$$
\alpha_p(n,L)=\left(\sum_{t=0}^{\mu_p}\left(\frac{(-1)^{\frac{k}2}dL}p\right)^tp^{\frac{(k-2)(t-\mu_p)}{2}}\right)\left(1-\left(\frac{-dL}{p}\right)\frac{1}{p}\right),
$$
and if $k$ is odd, then
$$
\alpha_p(n^2,L)=\frac{(1-p^{-1})(1-p^{-(k-2)\mu_p})}{p^{k-2}-1}+\left(\frac{(-1)^{\frac{k-1}{2}}dL}{p}\right)p^{-(k-2)\mu_p-\frac{k-1}{2}}+1,
$$
by Theorem 3.1 of \cite{y}.
The lemma follows directly from this.
\end{proof}

\begin{cor}\label{sregular}
Let $L$ be a $\z$-lattice of rank $k$. 
Then $L$ is strongly $s$-regular if and only if  $L$ satisfies the following equation
$$
\frac{r(n_1^2n_2^2,L)}{r(n_1^2,L)}= \frac{r(n_1^2n_2^2,\gen(L))}{r(n_1^2,\gen(L))},
$$
for any positive integer $n=n_1n_2$ satisfying all conditions given in Definition \ref{almost} and $r(n_1^2,L)\ne 0$.
\end{cor}
\begin{proof}
The corollary follows directly from Lemma \ref{min}.
\end{proof}

\section{Strongly $s$-regular binary $\z$-lattices}

In this section, we prove that there are only finitely many isometry classes of strongly $s$-regular binary  quadratic forms  if the minimum of the nonzero squares that are represented by them is fixed. 

Let $D$ be a negative integer that is congruent to $0$ or $1$ modulo $4$. 
We define $\mathfrak G_{D}$ to be the set of all proper classes of primitive binary quadratic forms of a fixed discriminant $-D$. Recall that the discriminant of a primitive binary quadratic form 
$$
f(x,y)=ax^2+bxy+cy^2,  \  \ (a,b,c)=1
$$
 is defined by $df=4ac-b^2$.
Hence $-df$ is congruent to $0$ or $1$ modulo $4$. It is well known that $\mathfrak G_{D}$ forms a finite abelian group under the composition law (for details, see \cite {ca}). For a proper class $\mathfrak A\in \mathfrak G_{D}$, we denote $r(n,\mathfrak A)$ to be $r(n,f)$ for some primitive binary quadratic form $f\in \mathfrak A$.  Let $L=\z u+\z v$ be a binary $\z$-lattice. The corresponding binary quadratic form $f_L$ to the $\z$-lattice $L$ is defined by $f_L(x,y)=Q(u)x^2+2B(u,v)xy+Q(v)y^2$.  Though the isometry class of $f_L$ is independent of the choice of the basis for $L$, the proper isometry class of $f_L$ depends on the choice of the basis for $L$. In fact, there might be two different proper isometry classes of binary quadratic forms corresponding to the binary $\z$-lattice $L$ according to the choice of the basis for $L$.

Let $L$ be an almost strongly $s$-regular binary $\z$-lattice.
Since we are assuming that the genus of $L$ represents at least one square of an integer, we always assume that $V_p \simeq \langle 1,dV_p \rangle$ for any prime $p$.  Hence  $S_p(V)=(-1,d(V_p))_p$ for any prime $p$. Therefore, we have the followings:
\begin{itemize}
\item [(i)]if $1$ is represented by $L_2$, then $L_2$ is isometric to one of
$$
\begin{pmatrix}1&\frac12\\\frac12&1 \end{pmatrix},\ \ \begin{pmatrix}0&\frac12\\\frac12&0 \end{pmatrix},\ \  \text{and} \ \ \langle 1, \e 2^t\rangle,
$$ 
for some $\e \in \z_2^{\times}$ and  an integer $t\ge 0$.

\item [(ii)]if $1$ is not represented by $L_2$, then $L_2$ is isometric to one of
$$
\hspace{1cm} \langle 3, \e_1 2^{2t_1} \rangle,\ \ \langle 3, \e_2 2^{2t_2+1} \rangle,\ \ \langle 5, \e_3 2^{2t_3+2} \rangle,\ \ \text{and} \ \ \langle 7, \e_4 2^{t_4+2} \rangle, 
$$
for some $\e_1,\e_4 \in \{1,5\}$, $\e_2\in \{3,7\}$ and $\e_3 \in \{1,3,5,7\}$ and  an integer $t_i\ge 1$  for $i=1,2,3,4$.

\item [(iii)]if $p$ is odd and $1$ is represented by $L_p$, then $L_p$ is isometric to $\langle 1, \e p^t\rangle$ for some $\e\in \z_p^{\times}$ and an integer $t\ge 0$.

\item [(iv)]if $p$ is odd and $1$ is not represented by $L_p$, then $L_p$ is isometric to $\langle \Delta_p, \e p^{2t}\rangle$ for some $\e\in \z_p^{\times}$ and  an integer $t\ge 1$.
\end{itemize}
Note that if $\ord_p\left(\frac14dL\right) \le 1$, then $L_p$ always represents $1$ under the assumption given above.

\begin{lem}\label{genus}
Any almost strongly $s$-regular binary $\z$-lattice represents all squares of integers that are represented by its genus. 
\end{lem}
\begin{proof}
Let $L$ be an almost strongly $s$-regular binary $\z$-lattice with respect to a finite set $P$ of primes. Suppose that a square of an integer $a^2$ is represented by the genus of $L$, but  is not represented by $L$.  Let $q$ be any prime such that $(q,2dL\cdot\prod_{p\in P}p)=1$, and let  $a=q^tb$, where $b$ is an integer such that $(q,b)=1$. Then, by Lemma \ref{tec},  we have
\begin{equation}\label{binary1}
r(q^2a^2,L)=r(q^{2t+2}b^2,L)=r(b^2,L) \cdot \sum_{i=0}^{2t+2}\left(\frac{-dL}{q}\right)=0.
\end{equation}
Since $a^2$ is represented by the genus of $L$, there is a binary $\z$-lattice $L' \in \gen(L)$ that represents $a^2$, that is, $r(a^2,L')\ne 0$. Let $f_L$($f_{L'}$) be the quadratic form corresponding to $L$($L'$, respectively). We also fix  a proper class $[f_L]$($[f_{L'}]$) containing the binary quadratic  form $f_L$($f_{L'}$, respectively).    Note that  the quadratic form corresponding to $L$ is uniquely determined up to isometry, but is not uniquely determined up to proper isometry, in general.  However, the result in the remaining part does not change, no matter which proper classes are chosen. Note that $[f_{L}], [f_{L'}] \in \mathfrak G_{-dL}$. Now, by Theorem 3.1 of Chapter $14$ in \cite{ca}, we have
$$
\gen(f_{L'})/\sim=\{[f_{L'}]\cdot[g]^2:[g]\in \mathfrak G_{-dL}\} \subset \mathfrak G_{-dL}.
$$
Here, $\gen(f_{L'})/\sim$ is the set of proper classes contained in the genus of $f_{L'}$.
Since $[f_L] \in \gen(f_{L'})/\sim$, there is a proper class $[g]\in \mathfrak G_{-dL}$ such that $[f_L]= [f_{L'}]\cdot[g]^2$. Furthermore, since the primitive binary quadratic form $g$ represents infinitely many primes, there is a prime $r$ with $(r,2dL\cdot\prod_{p\in P}p)=1$  that  is represented by the binary quadratic form $g$, that is, $r(r,g)\ne0$. Therefore, by the property of composition law (see \cite {ef}), we have
$$
r(a^2r^2,L)=r(a^2r^2,[f_L])=r(a^2r^2,[f_{L'}]\cdot[g]^2)\ne 0,
$$
which is a contradiction to \eqref{binary1}.
\end{proof}

\begin{cor}
Let $L$ be an almost strongly $s$-regular binary $\z$-lattice. Then any integer $m$ such that $m^2$ is represented by $L$ is a multiple of
$$
m_s(L)=\text{min}_{n\in \z^{+}}\{n:r(n^2,L)\ne 0\}.
$$
\end{cor}
\begin{proof} Let $m$ be an integer such that $m^2$ is represented by $L$. Then one may easily show that for any prime $p$, $\ord_p(m_s(L)) \le \ord_p(m)$ from the consequence of Lemma \ref{genus}. The corollary follows directly from this.
\end{proof}

Let $L$ be a binary $\z$-lattice. For any prime $p$, the $\Lambda_p$-transformation (or Watson transformation) is defined as follows: 
$$
\Lambda_p(L)= \{ x \in L : Q(x + z) \equiv Q(z) \  (\text{mod} \ p) \mbox{ for
all $z \in L$}\}.
$$
Note that $\Lambda_p(L)$ is, in fact, a sublattice of $L$.  Let $\lambda_p(L)$ be the primitive lattice (that is, $\mathfrak n(\lambda_p(L))=\z$) obtained from $\Lambda_p(L)$ by scaling $V=L\otimes \mathbb Q$ by a suitable
rational number. For a positive integer $N=p_1^{e_1}p_2^{e_2}\cdots p_k^{e_k}$,   we  also define 
$$
\lambda_N(L)=\lambda_{p_1}^{e_1}(\lambda_{p_2}^{e_2}(\cdots\lambda_{p_{k-1}}^{e_{k-1}}(\lambda_{p_k}^{e_k}(L))\cdots)).
$$
Note that $\lambda_p(\lambda_q(L))=\lambda_q(\lambda_p(L))$ for any primes $p \ne q$.

\begin{lem} \label{aniso}
Let $L$ be a binary $\z$-lattice. For any prime $p$ dividing $dL$, we have  
$$
r(pn,L)=r(pn,\Lambda_p(L)).
$$
\end{lem}

\begin{proof} See \cite {co}. \end{proof}

\begin{lem}\label{lambda}
Let $q$ be a prime and let $L$ be a binary $\z$-lattice such that $L_q$ does not represent $1$.
\begin{itemize}
\item [(i)]If $\ord_q\left(\frac14dL\right) \ge 3$, then $L$ is almost strongly $s$-regular with respect to $P$ if and only if $\lambda_q(L)$ is almost strongly $s$-regular with respect to $P$. Furthermore, if one of them is true, then $m_s(L)=q\cdot m_s(\lambda_q(L))$.

\item [(ii)]If $\ord_q\left(\frac14dL\right)=2$, then $L$ is almost strongly $s$-regular with respect to $P$ if and only if $\lambda_q(L)$ is almost strongly $s$-regular with respect to $P\cup \{q\}$. Furthermore, if one of them is true, then $m_s(L)=q\cdot m_s(\lambda_q(L))$.
\end{itemize}
\end{lem}
\begin{proof}
Since the proof is quite similar to each other, we only provide the proof of  the second case.
Note that if $q=2$ and $\ord_2\left(\frac14dL\right)=2$, then $L_2 \simeq \langle 3,4 \rangle$ or $\langle 3,20 \rangle$.
Suppose that $L$ is almost strongly $s$-regular with respect to a finite set of primes $P$. Let $n_1$ and $n_2$ be positive integers such that $n=n_1n_2$, and 
$$
\mathcal P(n_1)\subset \mathcal P(2dL)\cup P,\quad (n_2,2dL\cdot\prod_{p\in P}p)=1.
$$
Then we have
$$
r(q^2n_1^2n_2^2,L)=r(q^2n_1^2,L)\prod_{p\mid n_2}h_p(dL,\mu_p),
$$
where $\mu_p$ and $h_p(dL,\mu_p)$ are defined in Definition \ref{almost}. By Lemma \ref{aniso}, we have
\begin{equation}\label{2}
r(q^2n_1^2n_2^2,L)=r(n_1^2n_2^2,\lambda_q(L)) \quad \text{and} \quad r(q^2n_1^2,L)=r(n_1^2,\lambda_q(L)),
\end{equation}
which implies that 
$$
r(n_1^2n_2^2,\lambda_q(L))=r(n_1^2,\lambda_q(L))\prod_{p\mid n_2}h_p(dL,\mu_p).
$$
Since $\mathcal P(2dL) \cup  P= \mathcal P(2d(\lambda_q(L))) \cup P \cup \{q\}$, the above equation implies that
$\lambda_q(L)$ is almost strongly $s$-regular with respect to $P\cup \{q\}$.

Conversely, suppose that $\lambda_q(L)$ is almost strongly $s$-regular with respect to $P\cup \{q\}$. If $\text{ord}_q(n_1)\ge 1$, then by \eqref{2}, we have
$$
r(n_1^2n_2^2,L)=r(n_1^2,L)\prod_{p\mid n_2}h_p(dL,\mu_p).
$$
If $\text{ord}_q(n_1)=0$, then $r(n_1^2n_2^2,L)=r(n_1^2,L)=0$. Therefore, $L$ is almost strongly $s$-regular with respect to $P$.

Assume that either $L$ or $\lambda_q(L)$ is almost strongly $s$-regular. Note that from the assumption, $m_s(L)$ is divisible by $q$. Since $r(q^2n^2,L)=r(n^2,\lambda_q(L))$ by Lemma \ref{aniso}, we have
$m_s(L)=q\cdot m_s(\lambda_q(L))$.
\end{proof}

\begin{cor}\label{terminal}
Let $L$ be an almost strongly $s$-regular binary $\z$-lattice. Then there is a positive integer $N$ such that $\lambda_{N}(L)$ is almost strongly $s$-regular lattice such that $m_s(\lambda_{N}(L))=1$.
\end{cor}
\begin{proof}
The corollary follows directly from Lemmas \ref{genus} and \ref{lambda}.
\end{proof}

\begin{lem}\label{classnumber1}
If $L$ is an almost strongly $s$-regular binary $\z$-lattice such that $m_s(L)=1$, then $L$ has class number $1$.
\end{lem}
\begin{proof}
 Since $L$ is almost strongly $s$-regular with respect to a finite set $P$ of primes, for any prime $p$ with $p\notin \mathcal P(2dL) \cup P$, we have 
\begin{equation}\label{3}
r(p^2,L)=r(1,L)\cdot
\left(1+\left(\frac{-dL}{p}\right)+\left(\frac{-dL}{p}\right)^2\right).
\end{equation}
Suppose that $L$ has class number greater than $1$. Let  $L'\in \gen(L)$ be a binary $\z$-lattice that is not isometric to
 $L$.   Fix a proper class $[f_L]$ ($[f_{L'}]$) in $\mathfrak G_{-dL}$ corresponding to the binary $\z$-lattice $L$  ($L'$, respectively).
Since $L$ represent $1$,   we have, by Theorem 3.1 of Chapter $14$ in \cite{ca},
$$
\gen(f_{L})/\sim\,=\{[g]^2:[g]\in \mathfrak G_{-dL}\}.
$$
Hence, there is a proper class $[g] \in \mathfrak G_{-dL}$ such that $[f_{L'}]=[g]^2$. Since $g$ represents infinitely many primes, there is a prime $q$ with $(q,2dL\cdot\prod_{p\in P}p)=1$ that is represented by $g$, that is, $r(q,g)\ne 0$. Then by the property of the composition law, we have 
\begin{equation}\label{4}
r(q^2,[f_{L'}])=r(q^2,[g]^2)\ne 0.
\end{equation}
Now, by Theorem 4.1 of Chapter $12$ in \cite{h} and \eqref{3}, we have 
$$
\displaystyle \sum_{[h]\in \mathfrak G_{-dL}}r(q^2,[h])\displaystyle =r(1,[f_{L}])\left(1+\left(\frac{-dL}{q}\right)+\left(\frac{-dL}{q^2}\right)\right)
=r(q^2,[f_{L}]).
$$
This implies that $r(q^2,[f_{L'}])=0$, which is a contradiction to \eqref{4}.
\end{proof}

\begin{cor}
Let $L$ be a binary $\z$-lattice. Then $L$ is strongly $s$-regular if and only if $L$ is almost strongly $s$-regular.
\end{cor}
\begin{proof}
Note that the ``only if'' is trivial by Lemma \ref{almost1}. Suppose that $L$ is almost strongly $s$-regular.  By Corollary \ref{terminal} and Lemma \ref{classnumber1}, there is a positive integer $N$ such that  $\lambda_{N}(L)$ is of class number $1$ and then strongly $s$-regular. This implies that $L$ is strongly $s$-regular by Lemma \ref{lambda}.
\end{proof}

\begin{thm}
For any positive integer $m$, there are only finitely many isometry classes of strongly $s$-regular binary $\z$-lattices $L$ such that $m_s(L)=m$.
\end{thm}
\begin{proof}
Let $L$ be a strongly $s$-regular binary $\z$-lattice with $m_s(L)=m$.  By Corollary \ref{terminal} and Lemma \ref{classnumber1}, there is a positive integer $N$ such that $\lambda_{N}(L)$ has class number $1$. It is well known that there are only finitely many isometry classes of binary $\z$-lattices with class number $1$ (see \cite{ka} and \cite{lk}). Therefore the theorem follows directly from the fact that for any binary $\z$-lattice $K$ and any fixed prime $p$, there are finitely many $\z$-lattices whose $\lambda_p$-transformation is isometric to $K$.
\end{proof}

\begin{rmk}
{\rm (i) In fact, there are infinitely many strongly $s$-regular binary $\z$-lattices. For example, since both $\langle 1,2 \rangle$ and $\langle 2,9 \rangle$ have class number one, $\langle 2, 3^{2t}\rangle$ is strongly $s$-regular for any integer $t\ge 0$ by Lemma \ref{lambda}. 

\noindent (ii) In general, if a strongly $s$-regular binary $\z$-lattice $L$ does not represent $1$, then it is possible that $L$ has class number greater than $1$. For example, in (i), the binary $\z$-lattice $\langle 2, 3^{2t}\rangle$ is strongly $s$-regular with class number greater than $1$ for any integer $t\ge 2$. }
\end{rmk}

\section{Strongly $s$-regular $\z$-lattices with rank greater than two}

In this section, we will consider  strongly $s$-regular $\z$-latices with rank greater than $2$. We prove that for any integer $k\ge 3$, there are only finitely many isometry classes of strongly $s$-regular $\z$-lattices with rank $k$ if the minimum of the nonzero squares that are represented by them is fixed. In fact, this was proved in \cite{ko1} when $k=3$, and in \cite{kkm} when $k=4$.  The method used in \cite {kkm} and \cite {ko1} cannot be extended to the general rank case.  Since our method that we are using in this article can be applied to prove these two cases,  we proceed to our argument under the assumption that $k\ge 3$.   

\begin{lem}\label{genus1}
Any strongly $s$-regular $\z$-lattice represents all squares of integers that are represented by its genus. 
\end{lem}
\begin{proof}
Let $L$ be a strongly $s$-regular $\z$-lattice quadratic form of rank $k\ge 3$. 
Suppose that there is an integer $a$ such that $a^2$ is represented by the genus of $L$, but  is not represented by $L$ itself. Then for any prime $p\nmid 2dL$ and any integer $s\ge 1$, we have
$$
r(p^{2s}a^2,L)=r(p^{2s+2t}b^2,L)=r(b^2,L)h_p(dL,s+t)=0,
$$
where $a=p^t\cdot b$ for some integer $b$ with $(b,p)=1$ and $h_p$ is defined in Definition \ref{almost}.
By repeating this for any prime $q$ such that $(q,2dL)=1$, we may conclude that
\begin{equation}\label{e1}
r(n^2a^2,L)=0,
\end{equation}
for any integer $n$ with $(n,2dL)=1$.

Since we are assuming that $a^2$ is represented by the genus of $L$, there is a $\z$-lattice $L'\in \gen(L)$ such that $r(a^2,L')\ne0$.  If $k=3$, then the Class Linkage  Lemma in \cite {hjs} says that there is a prime $q\nmid 2dL$ such that $qL' \subseteq L$. Hence $q^2a^2$ is represented by $L$, that is, $r(q^2a^2,L)\ne 0$.  This is a contradiction to \eqref{e1}. Assume that $k\ge 4$. Since $a^2$ is represented by the genus of $L$, there exists a sufficiently large integer $m$ with $(m,2dL)=1$ such that
$$
r(m^2a^2,L)\ne 0,
$$
by Theorem 6.3  for the case when $k=4$, and Theorem 6.4  for the case when $k\ge 5$ of \cite{ha}. This is a contradiction to  \eqref{e1}.
\end{proof}

\begin{cor}
Let $L$ be a strongly $s$-regular $\z$-lattice. Then any integer $m$ such that $m^2$ is represented by $L$ is a multiple of
$$
m_s(L)=\text{min}_{n\in \z^{+}}\{n:r(n^2,L)\ne 0\}.
$$
\end{cor}
\begin{proof}
The corollary follows directly from Lemma \ref{genus1}.
\end{proof}

Let $L$ be a $\z$-lattice of rank $k$. For any positive integer $j\le k$, the $j$-th successive minimum of $L$ will be denoted by $\mu_j(L)$. For the definition of the successive minima, see Chapter 12 of \cite{ca}. It is well-known that $dL\le 2^k \mu_1(L)\mu_2(L)\cdots \mu_k(L)$ (for the proof, see Proposition 2.3 of \cite{ea}).

\begin{lem}\label{bound}
Let $L$ be a $\z$-lattice of rank $k$ and let $n$ be a positive integer. Then we have
$$
r(n,L)\le \left( \frac{2^{\frac{3k-1}{2}}\gamma_k^{\frac 12}}{dL^{\frac{k-1}{2k}}}\right)\cdot n^{\frac{k-1}{2}}+\sum_{i=1}^{k-1}t_i(k) n^{\frac{k-1-i}2},
$$
where  $t_i(k)=2^{k}\binom{k-1}{i}C_5(k)^{\frac{k-1-i}2}$  for any $i=1,2,\dots,k-1$. Here, $C_5(k)$ is an absolute constant depending only on $k$, which is defined in Theorem 3.1 of Chapter 12 in \cite{ca}.
\end{lem}
\begin{proof}
Let $L$ be a $\z$-lattice of rank $k$ and let $f_L$ be the quadratic form corresponding to $L$. Let 
$$
\begin{array}{ll}
f_L(x_1, \cdots, x_k)&\displaystyle=\sum_{1\le i,j\le k}a_{ij}x_ix_j \\[0.15cm]
                     &=h_1(x_1+c_{12}x_2+\cdots + c_{1k}x_k)^2 \\[0.15cm]
                     &\hspace{0.4cm}+h_2(x_2+c_{23}x_3+\cdots+c_{2k}x_k)^2\\[0.15cm]
                     &\hspace{0.4cm}+\cdots +h_kx_k^2
\end{array}
$$
be Minkowski-reduced with successive minima $\mu_1(L),\dots, \mu_k(L)$. Here, one may easily check that $dL=2^kh_1h_2\cdots h_k$. By Theorem 3.1 of Chapter 12 in \cite{ca}, there is an absolute constant $C_5(k)$ depending only on $k$ such that for any $i$ with $1\le i \le k$,
\begin{equation}\label{c5}
\frac{1}{C_5(k)}  \le  \frac{\mu_i(L)}{C_5(k)} \le h_i.
\end{equation}

Now, assume that $f_L(x_1, \cdots, x_k)=n$. Then, for any $i$ with $2\le i \le k$, we have
$$
(x_i+u_i)^2\le \frac{n}{h_i},
$$
where $u_i=c_{ii+1}x_{i+1}+\cdots+c_{ik}x_k$ for any $i$ with $2\le i\le k-1$ and $u_k=0$. This implies that
$$
-\sqrt{\frac{n}{h_i}} \le x_i+u_i \le \sqrt{\frac{n}{h_i}}.
$$
Since $x_i$ is an integer, we have for any $i$ with $2\le i \le k$,
\begin{equation}\label{bound1}
-\left[\sqrt{\frac{n}{h_i}}+u_i\right] \le x_i \le \left[\sqrt{\frac{n}{h_i}}-u_i\right],
\end{equation}
where $[x]$ is the greatest integer not exceeding $x$. Note that
$$
\left[\sqrt{\frac{n}{h_i}}-u_i\right]+\left[\sqrt{\frac{n}{h_i}}+u_i\right]\le 2\left[\sqrt{\frac{n}{h_i}}\right]+1,
$$
for any $i$ with $2\le i \le k$.
If the variables $x_2, x_3,\cdots, x_k$ are determined, then the number of possible $x_1$'s satisfying $f_L(x_1, \cdots, x_k)=n$ is at most $2$. Hence by \eqref{c5} and \eqref{bound1}, we have
$$
\begin{array}{ll}
r(n,L)=r(n,f_L) &\!\! \displaystyle \le 2\cdot\left(2\left[\sqrt{\frac{n}{h_2}}\right]+2\right) \cdots \left(2\left[\sqrt{\frac{n}{h_{k-1}}}\right]+2\right) \left(2\left[\sqrt{\frac{n}{h_k}}\right]+2\right)\\[0.4cm]

                &\!\! \displaystyle \le\frac{2^{k}\cdot n^{\frac{k-1}2}}{\sqrt{h_2\cdots h_{k-1}h_k}}+\sum_{i=1}^{k-1}2^{k}\binom{k-1}{i}C_5(k)^{\frac{k-1-i}2} \cdot n^{\frac{k-1-i}2}.
\end{array}
$$
Now, by Theorem 2.1 of Chapter 12 in \cite{ca}, we have
$$
\frac{1}{\sqrt{h_2\cdots h_{k-1}h_k}}=\frac{2^{\frac k2}\cdot \sqrt{h_1}}{\sqrt{dL}}\le \frac{2^{\frac{k-1}{2}}\gamma_k^{\frac 12} dL^{\frac{1}{2k}}}{\sqrt{dL}}=\frac{2^{\frac{k-1}{2}}\gamma_k^{\frac 12}}{dL^{(\frac12-\frac{1}{2k})}},
$$
where $\gamma_k$ is the Hermite's constant depending only on $k$. Therefore, we have
$$
r(n,L)\le \left(\frac{2^{\frac{3k-1}{2}}\gamma_k^{\frac 12}}{dL^{\frac{k-1}{2k}}}\right)\cdot n^{\frac{k-1}{2}}+\sum_{i=1}^{k-1}t_i(k) n^{\frac{k-1-i}2},
$$
where $t_i(k)=2^{k}\binom{k-1}{i}C_5(k)^{\frac{k-1-i}2}$ is an absolute constant depending only on $k$. This completes the proof.
\end{proof}

\begin{lem} \label{gen-tec} Let $M_0, N_0$ and $k$ be positive integers. Let $d$ be a positive integer satisfying the following property: for any prime $p \ge M_0$ not dividing $d$, $d<N_0\cdot p^{2k}$. Then there is an absolute constant depending only on $M_0, N_0$ and $k$  such that
$$
d<C_0(k,M_0,N_0).
$$
\end{lem}

\begin{proof}  Let $q_r$ be the $r$-th  smallest prime greater than or equal to $M_0$ so that $M_0 \le q_1<q_2<q_3<\cdots$. Let $r_0\ge2k+1$ be an integer such that  
$$
q_{r_0-2k}>2^{k(2k+1)}\cdot N_0.
$$ 
Then, by  Bertrand-Chebyshev theorem, we have
$$
q_1q_2\cdots q_{r_0}>N_0\cdot q_{r_0+1}^{2k}.
$$ 
Furthermore, we also have $q_1q_2\cdots q_{r}>N_0\cdot q_{r+1}^{2k}$ for any $r\ge r_0$.  

Let $s$ be the smallest integer such that $q_s$ does not divide $d$. If $s > r_0$, then we have 
$$
d>q_1q_2\cdots q_{s-1}>N_0\cdot q_s^{2k},
$$ 
which is a contradiction to the assumption.  Therefore $s \le r_0$ and 
$$
d<N_0\cdot q_{r_0}^{2k}=:C_0(k,M_0,N_0).
$$
This completes the proof.  \end{proof}

\begin{thm} \label{main}
Let $k$ be an integer greater than $2$. 
For any positive integer $m$, there are only finitely many isometry classes of strongly $s$-regular $\z$-lattices $L$ of rank $k$  such that $m_s(L)=m$.
\end{thm}

\begin{proof}
Let $L$  be a strongly $s$-regular $\z$-lattice of rank $k$ such that $m_s(L)=m$. 
Let $y_i \in L$ be a vector such that $Q(y_i)=\mu_i(L)$, where $\mu_i(L)$ is the $i$-th successive minimum of $L$. Define a sublattice $L(k-1)=(\q y_1+\q y_2+\dots+\q y_{k-1}) \cap L$
of $L$.
Since $r(n,L(k-1))$ is the Fourier coefficient of a modular form of weight $\frac{k-1}{2}$ , it is well known 
that (see, for example, \cite{lo})
$$
r(n,L(k-1))\in
\begin{cases}
O\left(n^{\frac34}\right) \ \ \text{if $k=4$},\\[0.2cm]
 O\left(n^{\frac{k-3}{2}+\e}\right) \ \ \text{otherwise},
\end{cases}
$$
for any $\e>0$. Hence, if $d(L(k-1)) < 2^{\frac{(3k-4)(k-1)}{k-2}} \cdot \gamma_{k-1}^{\frac{k-1}{k-2}} \cdot m^{2(k-1)}$, then there exists an absolute constant $C_{k,m}$ depending only on $k$ and $m$ such that 
$$
r(n,L(k-1)) \le C_{k,m} \cdot n^{e_k}
$$
for any positive integer $n$, where $e_k=\frac 34$ if $k=4$, $e_k=\frac{2k-5}{4}$ otherwise.        
We define
$$
M_0\!=\!
\text{min} \left\{\!\!\!\!\begin{array}{l|l} & 2(M^{k-2}-M^{\frac{k-2}{2}}+1) > M^{k-2}+\sum_{i=1}^{k-2}t_i(k-1) (mM)^{k-2-i} \\[-0.23cm]
M\in \mathbb N &\\[-0.23cm]
&    \text{and}\quad 2(M^{k-2}-M^{\frac{k-2}{2}}+1) > C_{k,m} \cdot m^{2e_k}\cdot M^{2e_k}\end{array} \!\!\!\right\},
$$
where $t_i(k-1)=2^{k-1}\binom{k-2}{i}C_5(k-1)^{\frac{k-2-i}2}$ is an absolute constant depending only on $k$. Note that $M_0$ is an absolute constant depending only on $k$ and $m$.

Let $p$ be any prime greater than or equal to $M_0$ such that $p$ does not divide $2dL$. 
Since $L$ is strongly $s$-regular, we have
$$
r(m^2p^2,L)=r(m^2,L)h_{p}(dL,1),
$$
where 
$$
h_{p}(dL,1)=
\begin{cases}
\displaystyle p^{k-2}+\left(\frac{(-1)^{\frac k2}dL}{p}\right)p^{\frac{k-2}{2}}+1 \ \ \text{if $k$ is even,}\\
\displaystyle p^{k-2}+1-p^{\frac{k-3}{2}}\left(\frac{(-1)^{\frac {k-1}{2}}dL}{p}\right) \ \ \text{otherwise}.
\end{cases}
$$
Therefore we have 
$$
r(m^2p^2,L)=r(m^2,L)h_{p}(dL,1) \ge 2(p^{k-2}-p^{\frac{k-2}{2}}+1).
$$
On the other hand, if $d(L(k-1)) < 2^{\frac{(3k-4)(k-1)}{k-2}} \cdot \gamma_{k-1}^{\frac{k-1}{k-2}} \cdot m^{2(k-1)}$, then we have
$$
r(m^2p^2,L(k-1)) \le  C_{k,m} \cdot m^{2e_k}\cdot p^{2e_k},
$$
and
if $d(L(k-1)) \ge 2^{\frac{(3k-4)(k-1)}{k-2}} \cdot \gamma_{k-1}^{\frac{k-1}{k-2}} \cdot m^{2(k-1)}$, then by Lemma \ref{bound}, we have
$$
r(m^2p^2,L(k-1)) \le p^{k-2}+\sum_{i=1}^{k-2}t_i(k-1) (mp)^{k-2-i}.
$$
Since $p$ is greater than or equal to $M_0$, we have
$$
r(m^2p^2,L(k-1)) < r(m^2p^2,L).
$$
This implies that 
$$
\mu_k(L) \le m^2p^2 \quad  \text{and}   \quad dL \le 2^km^{2k}p^{2k}.
$$ 
Now, the theorem follows directly from Lemma \ref{gen-tec}.  \end{proof}


\begin{thebibliography}{abcd}

\bibitem {ca} J. K. S. Cassels, {\em Rational quadratic forms}, Academic Press, 1978.

\bibitem {cl} S. Cooper and H.-Y. Lam, {\em On the diophantine equation $n^2=x^2+by^2+cz^2$}, J. Number Theory \textbf{133}(2013), 719-737. 

\bibitem {co} W. K. Chan and B.-K. Oh, {\em Finiteness theorems for positive definite $n$-regular quadratic forms}, Trans. Amer. Math. Soc. \textbf {355}(2003), 2385-2396.

\bibitem {ea} A. G. Earnest, {\em The representation of binary quadratic forms by positive definite quaternary quadratic forms}, Trans. Amer. Math. Soc. \textbf {345}(1994), 853-863.


\bibitem {ef} A. G. Earnest and R. W. Fitzgerald, {\em Represented value sets for integral binary quadratic forms and lattices}, Proc. Amer. Math. Soc. \textbf {135}(2007), 3765-3770.



\bibitem {gpq} X. Guo, Y. Peng, and H. Qin, {\em On the representation numbers of ternary quadratic forms and modular forms of weight 3/2}, J. Number Theory \textbf{140}(2014), 235-266.


\bibitem {ha} J. Hanke, {\em Local densities and explicit bounds for representability by a quadratic form}, Duke Math. J. \textbf{124}(2004), 351-388.


\bibitem {hjs} J. S. Hsia, M. J\"ochner, and Y. Y. Shao, {\em A structure theorem for a pair of quadratic forms}, Proc. Amer. Math. Soc. \textbf {199}(1993), 731-734.


\bibitem {h} L. K. Hua, {\em Introduction to Number Theory}, Springer, Berlin, 1982.   


\bibitem {hur}  A. Hurwitz,  {\em  L'Interm\'ediaire des Math\'ematiciens} \textbf{14}(1907), 107.


\bibitem{hu} W. H\"urlimann, {\em Cooper and Lam's conjecture for generalized Bell ternary quadratic forms}, J. Number Theory \textbf{158}(2016), 23-32.



\bibitem {ka} E. Kani, {\em Idoneal numbers and some generalizations}, Ann. Sci. Math. Qu\'ebec \textbf{35}(2011), 197-227.


\bibitem {kkm} K. Kim, {\em The number of representations of squares by integral quaternary quadratic forms}, submitted.

\bibitem {ko1} K. Kim and B.-K. Oh, {\em The number of representations of squares by integral ternary quadratic forms}, J. Number Theory \textbf{180}(2017), 629-642.


\bibitem {ko2} K. Kim and B.-K. Oh, {\em The number of representations of squares by integral ternary quadratic forms (II)}, J. Number Theory \textbf{173}(2017), 210-229.


\bibitem {ki} Y. Kitaoka, {\em Arithmetic of quadratic forms}, Cambridge University Press, 1993.


\bibitem {lo} R. J. Lemke Oliver, {\em Eta-quotients and theta functions}, Adv. Math. \textbf{241}(2013), 1-17.

\bibitem {lk} D. Lorch and M. Kirschmer, {\em Single-class genera of positive integral lattices}, LMS J. Comput. Math. \textbf{16}(2013), 172-186.


\bibitem {ol}  C. D. Olds, {\em On the representations, $N_3(n^2)$}, Bull. Amer. Math. Soc. \textbf{47}(1941), 499-503.


\bibitem {om} O. T. O'Meara, {\em Introduction to quadratic forms}, Springer Verlag, New York, 1963.



\bibitem {y} T. Yang, {\em An explicit formula for local densities of quadratic forms}, J. Number Theory \textbf {72}(1998), 309-356.

\end{thebibliography}
\end{document}